\documentclass[11pt]{amsart}
\usepackage{amsmath,amssymb,amsfonts,amscd,epsfig}

\theoremstyle{plain}
\newtheorem{thm}{Theorem}
\newtheorem{lem}[thm]{Lemma}
\newtheorem{con}[thm]{Conjecture}
\newtheorem{cor}[thm]{Corollary}
\newtheorem{prop}[thm]{Proposition}
\newtheorem{remark}[thm]{Remark}
\newtheorem{defn}[thm]{Definition}

\newcommand{\A}{{\mathbb A}}

\newcommand{\FF}{{\mathbb F}}

\newcommand{\NN}{{\mathbb N}}

\newcommand{\QQ}{{\mathbb Q}}
\newcommand{\RR}{{\mathbb R}}

\newcommand{\ZZ}{{\mathbb Z}}

\title{Modular forms in Quantum Field Theory}
\author{Francis Brown and Oliver Schnetz}
\begin{document}
\begin{abstract}
The amplitude of a Feynman graph in Quantum Field Theory is related to the point-count over finite fields of the corresponding graph hypersurface. 
This article reports on an experimental study of point counts over $\FF_q$ modulo $q^3$,
for graphs  up to loop order 10. It is found that many of them 
are given by  Fourier coefficients of  modular forms of weights $\leq 8$ and levels $\leq17$. 
\end{abstract}
\maketitle
\section{Introduction}
We first explain the definition of the $c_2$-invariant of a graph and its connection to its period. Then we turn to modularity and our results.
\subsection{The $c_2$-invariant}
Let $G$ be a connected graph. The graph polynomial of $G$ is defined by associating a variable $x_e$ to every edge $e$ of $G$ and setting
\begin{equation}\label{1}
\Psi_G(x)=\sum_{T\,\rm span.\,tree}\;\prod_{e\not \in T}x_e,
\end{equation}
where the sum is over all spanning trees $T$ of $G$. These polynomials first appeared in Kirchhoff's work on currents in electrical networks \cite{KIR}.

The set of finite graphs $G$ is filtered by the maximal degree of the set of  vertices of $G$. We say that 
\begin{equation} 
G \hbox{  is in } \phi^n  \, \hbox{theory}   \hbox{ if }   \deg(v) \leq n  \, \hbox{ for all vertices } v \hbox{ of } G
\end{equation}
and we will mostly restrict to the physically meaningful case  of $\phi^4$.

The arithmetic content of perturbative Quantum Field Theories
is given by  integrals of rational functions, whose denominator is the square of the graph polynomial.
This requires  a convergency condition for the graphs. A connected graph $G$  is called primitive-divergent (see Fig.\ 1) if
\begin{eqnarray} \label{defnprimdiv}  N_G  & = & 2h_G,  \\
N_{\gamma} &> &2 h_{\gamma} \quad \hbox{ for all non-trivial strict subgraphs } \gamma \subsetneq G \ ,\label{defnprimdiv1}
\end{eqnarray}
where $h_{\gamma}$ denotes the number of loops (first Betti number) and $N_{\gamma}$ the number of edges in a graph $\gamma$.
It is easy to see that primitive-divergent graphs with at least three vertices are simple: they have no multiple edges or self-loops.

\begin{figure}[ht]
\epsfig{file=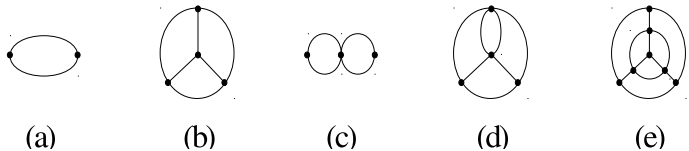,width=\textwidth}
\caption{Graphs (a) and (b) are primitive-divergent whereas graphs (c), (d), and (e) have subdivergences.}
\end{figure}

If $G$ is primitive-divergent, the period of $G$ is defined by the  convergent  integral \cite{BEK}, \cite{BROWN}   (which is independent of the choice of edge $N_G$)
\begin{equation}\label{2}
P(G)=\int_0^\infty\cdots\int_0^\infty\frac{{\rm d}x_1\cdots{\rm d}x_{N_G-1}}{\Psi_G(x)^2|_{x_{N_G}=1}}\in\RR_+. 
\end{equation}
In this way,  $P(G)$ defines a map from the set of primitive-divergent graphs to positive real numbers. In  the case of $\phi^4$ theory they  are renormalization-scheme
independent contributions to the $\beta$-function \cite{IZ}.

Since $(\ref{1})$ is defined over the integers, it defines an affine scheme  of finite type over $\mathrm{Spec}\, \ZZ$ which is called  the graph hypersurface $X_G \subset  \A^{N_G}$.
For any field $k$, we can therefore consider the zero locus $X_G(k)$ of $\Psi_G$ in $k^{N_G}$.
If the ground field $k\cong \FF_q$ is finite, we have the point-counting function
\begin{equation}\label{2a}
[X_G]_q : q \mapsto \# X_G(\FF_q) \in \NN.
\end{equation}
It defines a map from prime powers to non-negative integers.
Inspired by the appearance of multiple zeta values  in the period integral \cite{BK}, Kontsevich informally conjectured in 1997 that the function $[X_G]$ might be polynomial
in $q$ for all graphs \cite{KONT}. Although the conjecture is true for small graphs \cite{STEM} and for certain sets of nearly complete graphs \cite{STA}, \cite{CY}, it is false in general.
In \cite{BB} Belkale and Brosnan used Mn\"ev's universality theorem to prove that the $[X_G]$ for arbitrary graphs are, in a certain sense, of general type.
In \cite{BROWN}, \cite{BS} we proved the conjecture is true for some infinite families of graphs in $\phi^4$ theory.

Nonetheless a connection between the point-counting function and the period  $(\ref{2})$ remains valid in all cases. Recent work \cite{SchnetzFq}, \cite{BS}, \cite{BD}
shows that  certain information about the period is indeed detected by a small piece of the point-counting function $[X_G]_q$, called the $c_2$-invariant. 
\begin{prop}\label{prop0} \cite{BS}
If $G$ has at least three vertices  there exists a quantity
$$ c_2(G) = (c_2(G)_q)_{q} \in \prod_{q} \ZZ / q \ZZ$$
where $q$ ranges over the set of prime powers  and $c_2(G)_q \in \ZZ/q \ZZ$ is defined as follows. One shows that $[X_G]_q \equiv 0 \mod q^2$, and sets
\begin{equation}\label{2b}
c_2(G)_q \equiv [X_G]_q q^{-2} \mod q.
\end{equation}
\end{prop}
In the case when $[X_G]_q$ is a polynomial in $q$, the $c_2$-invariant is simply the reduction mod $q$ of  the coefficient of $q^2$  in this polynomial, and so  there is an integer $M$ such that
$c_2(G)_q \equiv M \mod q$ for all $q$. When this happens, we say that the $c_2$-invariant is constant. In other words, the $c_2$-invariant is said to be constant if and only if it  is in the image of the map
$$\ZZ \rightarrow \prod_{q} {\ZZ / q \ZZ}$$
Any graph $G$ with a non-constant $c_2$-invariant is therefore a counter-example
to Kontsevich's problem. The connection between the period and the $c_2$-invariant is further borne out by the following conjecture, which holds in all presently known examples: 
\begin{con}\label{con2}
If $P(G_1)=P(G_2)$ for primitive divergent graphs $G_1$, $G_2$, then $c_2(G_1)_q\equiv c_2(G_2)_q\mod q$ for all prime powers $q$.
\end{con}
This conjecture is supported by \cite{BD}, where it is shown that, for a large class of graphs,  the $c_2$-invariant is related to the  de Rham framing on the cohomology of the graph hypersurface
given by the integrand of (\ref{2}).

\subsection{Which motives for quantum field theory?} Since graph polynomials are not polynomially countable \cite{BB, BS}, and more to the point,  the period $(\ref{2})$  does not factorize
through a category of mixed Tate motives \cite{BD}, it follows from standard conjectures in transcendence theory  that the integral  $(\ref{2})$ will not be a multiple zeta value in general.
An important question is to try to ascertain which families of periods do occur as values of   $(\ref{2})$, especially when one places physically meaningful restrictions  on the graphs $G$. Due to the immense difficulty of computing the periods (\ref{2}) directly,  or even to obtain any non-trivial information about the  mixed Hodge structure or motive  \cite{BEK} underlying (\ref{2})
beyond 6 loops, one is forced to find new methods to probe the arithmetic content of $\phi^4$ theory at high loop orders.
The goal of this paper is to argue that the $c_2$-invariant gives an effective method to do just this, and  to  report  
on an experimental  study of all $c_2$-invariants of primitive $\phi^4$-graphs with up to and including $10$ loops.

Note that a naive approach to computing the point-counting functions $[X_G]_q$, for any $q$, is completely impossible: at ten loops
the graph polynomials $\Psi_G$ are of degree 10 in 20 variables, and have thirty to forty thousand monomials. Furthermore, there 
are several thousand primitive divergent graphs. The main point is that the $c_2$-invariant satisfies sufficiently many combinatorial properties to reduce 
this to a manageable computation. 

A first reduction, which uses  graph-theoretical arguments together with a weakened version of Conj.\ \ref{con2} (see Conj.\ \ref{con4})
allows us to reduce the number of relevant cases to 284 `prime ancestors' (see Thm.\ \ref{thm3}). The number of prime ancestors at a certain loop order is listed in Tab.\ 2.
A second, crucial,  reduction is Thm.\ \ref{thm4}, which reduces the $c_2$-invariant of the hypersurface $X_G$ to 
that of a hypersurface of much smaller dimension, for which the points can be counted for at least the first 6 primes in all  284 cases. 
This is enough to distinguish 145 $c_2$-invariants which are listed in Tab's.\ 7 and 8 at the end of the article. For graphs  up to 9 loops  we computed the $c_2$-invariants of
their prime ancestors for the first 12 primes (Tab.\ 6).

\subsection{Findings}
For small graphs (with $\leq 6$ loops \cite{STEM}) the $c_2$-in\-variant is constant and it is known that the integral $(\ref{2})$  is a linear combination of multiple zeta values \cite{BK}, \cite{CENSUS}.
At 7 loops, we find the first examples of $c_2$-invariants which are quasi-constant \cite{DORY1, SchnetzFq}. This means  that the $c_2$ is constant after  excluding finitely many primes, or by passing to a finite field extension  (Def.\ \ref{def7}). Experimentally, we find  that for graphs in $\phi^4$ theory up to 10 loops,  only the constants $c_2 =0,-1$, and  three other  quasi-constants can occur  corresponding to extending by  2nd, 3rd and 4th roots of unity (Conj.\ \ref{con5}).
All these examples correspond to Tate motives.

The first non-Tate examples occur at 8 loops.  We say that  a $c_2$-invariant is \emph{modular} if the   point-counts  $c_2(G)_p$  over finite fields $\FF_p$ where $p$ is prime,   coincide with the $p$th Fourier coefficients
of a modular form (possibly with a non-trivial character). In \cite{BS} we proved using modularity theorems for singular K3 surfaces that a certain graph with 8 loops is modular for all primes $p$.
No such theorem  is currently available for any other $c_2$-invariant. Therefore, we shall abusively say that a graph $G$ is modular if  $c_2(G)_p$ coincides with the Fourier coefficients of a modular form for small $p$. In practice, computing just a handful of primes is enough
to fix the character of $c_2$ uniquely: for instance, the likelihood of a false identification after counting points over the first 11 primes is of the order of one in 
 $| \FF_2 \times \ldots \times \FF_{31}| \sim 2 \times 10^{11}$. Thus we can be fairly confident that our  experimentally-modular graphs
(modular for at least the first 11 primes) are indeed modular for all primes.

Motivated by the 8-loop example, we searched for other modular examples of primitive-divergent graphs $G$ in $\phi^4$ theory up to and including 10 loops.
The output of this search is  summarized in Tab.\ 1. In total 16 out of the (at least) 145 $c_2$-invariants of graphs
with $\leq10$ loops in $\phi^4$ theory are modular for (at least) the first 11 primes. The modular forms that correspond to these $c_2$-invariants are in heavy boxes in Tab.\ 1.

One immediately notices from Tab. 1  that the modular forms coming from  $\phi^4$ graphs have very low levels and never have weight 2.
With our data we can rule out (assuming Conj.\ \ref{con4}) weight 2 modularity for all levels up to 1200 in graphs with $\leq$ 10 loops.
This  absence of weight 2 should be a general property of primitive $\phi^4$-graphs  (Conj.\ \ref{con6}). Moreover, we only found levels 7, 8, and 12 at weight 3. Remarkably, all these modular forms were already found in graphs with not more than 9 loops (the lower index in the boxes). At weight 4
we found modular forms of level 5, 6, 7, 13, and 17. At weight 5 only level 4 could be identified. At weight 6 we had modular forms of levels 3, 4, 7, 10, whereas at weights
7 and 8 levels 3 and 2, 5 could be found. We observe a tendency to higher weight for higher loop order. Otherwise we are not aware of any particular pattern in the identified modular forms.

In order to investigate the effect of the topology of the graph on the $c_2$-invariant, we also computed the $c_2$-invariants of all log-divergent graphs
(graphs which satisfy (\ref{defnprimdiv}) and (\ref{defnprimdiv1}), but which are not necessarily in $\phi^4$ theory, see Fig.\ 6) up to and including 9 loops.
These additional graphs give rise to the first three levels of weight 2 and fill in the gap at level 11 of weight 3. They are shown in light boxes in Tab.\ 1.
Moreover,
already at 8 loops there exists a non-$\phi^4$ graph with a quasi-constant $c_2$-invariant which does not occur in the set of  $\phi^4$ graphs with up to 10 loops.
Its point-count (\ref{24}) is given by the number of zeros of  $x^2+x-1$,  which is not cyclotomic.
In conclusion, it seems that the point-counting functions of `physical graphs' (i.e.\ in $\phi^4$) are highly constrained compared to the set of all graphs.
\vskip5mm

\begin{tabular}{l|lllllll}
weight&2&3&4&5&6&7&8\\\hline
level&
\setlength{\fboxrule}{0.1mm}\fbox{11}$_{\rm \phi^{>4},\,9}^\eta$&\setlength{\fboxrule}{0.4mm}\fbox{7}$_8^\eta$&\setlength{\fboxrule}{0.4mm}\fbox{5}$_8^\eta$&
\setlength{\fboxrule}{0.4mm}\fbox{4}$_9^\eta$&\setlength{\fboxrule}{0.4mm}\fbox{3}$_8^\eta$&\setlength{\fboxrule}{0.4mm}\fbox{3}$_9$&\setlength{\fboxrule}{0.4mm}\fbox{2}$_{10}^\eta$\\
&\setlength{\fboxrule}{0.1mm}\fbox{14}$_{\rm \phi^{>4},\,9}^\eta$&\setlength{\fboxrule}{0.4mm}\fbox{8}$_8^\eta$&\setlength{\fboxrule}{0.4mm}\fbox{6}$_9^\eta$&
7&\setlength{\fboxrule}{0.4mm}\fbox{4}$_9^\eta$&7&3\\
&\setlength{\fboxrule}{0.1mm}\fbox{15}$_{\rm \phi^{>4},\,9}^\eta$&\setlength{\fboxrule}{0.1mm}\fbox{11}$_{\rm \phi^{>4},\,9}$&\setlength{\fboxrule}{0.4mm}\fbox{7}$_{10}$&
8&5&8&\setlength{\fboxrule}{0.4mm}\fbox{5}$_{10}$\\
&17&\setlength{\fboxrule}{0.4mm}\fbox{12}$_9^\eta$&8&11&6&11&6\\
&19&15&9&12&\setlength{\fboxrule}{0.4mm}\fbox{7}$_9$&15&7\\
&20&15&10&15&8&15&8\\
&21&16&12&15&9&16&8\\
&24&19&\setlength{\fboxrule}{0.4mm}\fbox{13}$_9$&19&\setlength{\fboxrule}{0.4mm}\fbox{10}$_{10}$&19&9\\
&26&20&$\vdots$&20&10&20&10\\
&26&20&\setlength{\fboxrule}{0.4mm}\fbox{17}$_{10}$&20&10&20&12
\end{tabular}
\vskip5mm

\noindent Table 1: Newforms of low level with rational Fourier coefficients and their first appearance in $\phi^4$-theory (see main text).
The lower index indicates the lowest loop order at which they occur, an upper index $\eta$ indicates that the modular form is an  $\eta$-product, see Tab.\ 3 in Sect.\ \ref{results}.
A subscript $`\phi^{>4}, 9$' denotes the modular form of a graph which does not lie in $\phi^4$ theory (i.e.\ which has a vertex with valency greater than four) with 9 loops.
The table does not include any non-$\phi^4$ 10 loop graphs.
\vskip5mm

Besides the modular and (quasi-)constant examples, there are many $c_2$-invariants which we were unable to undentify. 
Up to 8 loops $\phi^4$ theory is  fully quasi-constant or modular, but  at 9 loops there are 10 $c_2$-invariants which are neither quasi-constant nor modular of low level.
Their sequences are listed for the first 12 primes in Tab.\ 6. At 10 loops we have another 114 unidentified sequences which are listed for the first 6 primes in Tab.\ 8.
In one case it is possible to reduce an unidentified $c_2$-invariant ($i_{101}$ in Tab.\ 8) to the (affine) point-count of a
4-fold which is the projective zero locus of the degree 6 polynomial (\ref{23}). The $i_{101}$ $c_2$-invariant is listed for the first 100 primes in Tab.\ 5, and the $c_2$-invariants
of 7 other  accessible cases are listed for the first 50 primes.

A surprising consequence of our findings is the following  trichotomy for graphs in $\phi^4$ theory. Up to 10 loops they 
fall into three categories:
\begin{enumerate}
\item Graphs $G$ with $c_2$-invariant equal to $-1$. These appear to have a unique prime ancestor: namely the wheel with 3 spokes (whose completion is the  graph  $K_5$). In other words, this entire class is generated 
\emph{by a single graph} by completion and double-triangle operations.
\item Graphs which have $c_2$-invariant equal to $0$. These graphs are expected to have weight-drop and therefore contribute to the perturbative expansion in a quite different way from 
the previous class.
\item Graphs with non-constant $c_2$-invariants. These start at 7 loops and are all counter-examples to Kontsevich's conjecture. This class contains the modular graphs described above.
\end{enumerate}
The last set of columns (loop order $\ell$) in Tab.\ 7 suggests that the wheel with 3 spokes  is the only graph in $\phi^4$ theory which plays a distinguished role.
In particular, there are many prime ancestors with vanishing $c_2$-invariant (the smallest of which is the $K_{3,4}$ graph at loop order 6). In \cite{BY}, several families of graphs with
vanishing $c_2$-invariant were constructed.
\\

\noindent{\bf Acknowledgements.} We are very grateful to Jonas  Bergstr\"om for correspondence 
on Siegel modular forms. The article was written while both authors were visiting scientists at Humboldt University, Berlin.
The computations were performed on the Erlanger RRZE Cluster. Francis Brown is partially supported by ERC grant 257638. 

\section{Equivalence classes of graphs}\label{GN}
\subsection{Completed primitive graphs}
The map $(\ref{2})$ from graphs to periods satisfies various identities  which we review here. 
Recall that a graph is $4$-regular if every vertex has degree $4$.  

\begin{defn}\label{def1} Let $\Gamma$ be a connected $4$-regular graph with at least 3 vertices. We say that $G$ is 
completed primitive if every 4-edge cut of $\Gamma$ is either connected,  or has a component consisting of  a single vertex.
\end{defn}

The simplest completed primitive graph is a 3-cycle of double edges. There exists no completed primitive graph with 4 vertices but there is
a unique completed primitive graph with 5 vertices, the complete graph $K_5$, and a unique one with 6 vertices, the octahedron $O_3$ (see Fig.\ 4).

Let $G$ be  primitive-divergent in $\phi^4$. From Euler's relation for a connected graph:
$N_G - V_G = h_G - 1$, combined with $(\ref{defnprimdiv})$, one easily shows 
that $G$ has  4 vertices of degree 3 or 2 vertices of degree 2.
Its \emph{completion} $\Gamma$  is defined to be the graph
obtained by adding a new vertex to $G$, and connecting it via a single edge to every 3-valent vertex in $G$, or a double
edge to every 2-valent vertex of $G$. One shows that the graph $\Gamma$ is completed primitive. The following lemma implies that every completed
primitive graph arises in this way (note that two non-isomorphic graphs can have the same completion).

\begin{lem} Let $\Gamma$ be a completed primitive graph. Then for any vertex $v\in \Gamma$, the graph $\Gamma \backslash v$ is primitive-divergent in $\phi^4$ theory.
\end{lem}
\begin{proof}
See the proof of proposition 2.6 in \cite{CENSUS}.
\end{proof}
An immediate consequence of the lemma is that completed primitive graphs with at least 5 vertices are simple.
The following proposition  states that the period of a primitive divergent graph in $\phi^4$ theory only depends on its completion.

\begin{prop} \label{propindepofv}
Let $\Gamma$ be a completed primitive graph, and let $G= \Gamma \backslash v$. Then the integral
(\ref{2}) corresponding to the graph $G$ converges, and  the value of the integral  is independent of the choice of vertex $v$.
\end{prop} 

\begin{proof}
See proposition 2.6 in \cite{CENSUS} and Lemma 5.1 and Prop.\ 5.2 in \cite{BEK}  for the convergence. Independence of $v$ follows from Theorem 2.7 (5) in \cite{CENSUS}.
\end{proof} 

Let $\Gamma$ be a completed primitive graph. By the previous proposition we can define the period of $\Gamma$ to be the real number
\begin{equation}\label{3}
P_\Gamma=P(\Gamma\backslash v)
\end{equation}
for any vertex $v\in\Gamma$.  Abusively, we denote the  loop order $\ell_\Gamma$ of a completed primitive graph $\Gamma$ to be the number of independent cycles in  any of its primitive-divergent graphs $\Gamma\backslash v$: 
\begin{equation}\label{6}
\ell_\Gamma=h_{\Gamma\backslash v}=V_\Gamma-2.
\end{equation}
The period of the 3-cycle of double edges is 1, 
$P_{K_5}=6\zeta(3)$, and $P_{O_3}=20\zeta(5)$. For a list of all completed primitive graphs up to loop order 8 and their known periods see \cite{CENSUS}.

\subsection{The product identity}
\begin{figure}[ht]
\epsfig{file=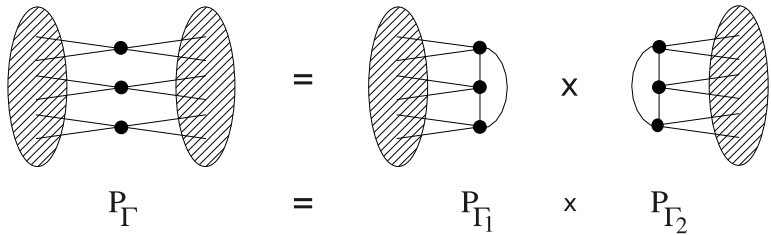,width=\textwidth}
\caption{Vertex-connectivity 3 leads to products of periods.}
\end{figure}

We say that   a  completed primitive graph $\Gamma$  is \emph{reducible}  if there exists a set of three vertices $v_1,v_2,v_3$  in $\Gamma$ such that
$\Gamma\backslash \{v_1,v_2,v_3\}$ has more than one connected component, see Fig.\ 2. A well-known feature of the period is the following factorization property.

\begin{prop}\label{thm1}
Every reducible completed primitive graph $\Gamma$ is isomorphic to  a graph  obtained by  gluing  two completed primitive graphs $\Gamma_1$ and $\Gamma_2$ on triangle faces
followed by the removal of the triangle edges (see Fig.\ 2). The period of $\Gamma$ is the product of the periods of $\Gamma_1$ and $\Gamma_2$,
\begin{equation}\label{7}
P_{\Gamma}=P_{\Gamma_1}P_{\Gamma_2}.
\end{equation}
\end{prop}
\begin{proof}See
Thm.\ 2.10  in \cite{CENSUS} or \S3.4 in \cite{BROWN}. 
\end{proof}

Note that the above  gluing operation is not defined for all pairs of graphs, since not all completed primitive graphs contain triangles. Let us define
$$\mathcal{G}= \langle \Gamma,  \cup  \rangle$$
to be the free commutative monoid generated by completed primitive graphs $\Gamma$ (with multiplication denoted by $\cup$), and  let 
$$\mathcal{G}_0 = \mathcal{G} / (\Gamma \hbox{ red.} \sim \Gamma_1 \cup \Gamma_2) $$
be the quotient by the equivalence relation generated by   identifying  a reducible completed primitive graph  $\Gamma$ 
with the union of  its components.

\begin{cor} The period gives a well-defined multiplicative map
$$
P: \mathcal{G}_0 \longrightarrow \RR_+.
$$
\end{cor}

\begin{remark}  There exist  two other known identities on periods: the twist identity \cite{CENSUS} and the (rather rare) Fourier identity \cite{BK}. It turns out that these
are often subsumed by the double triangle relation (below) at low loop orders and therefore we shall not include them in the present set-up.
\end{remark}
There are no presently-known identities on periods which relate a non-trivial linear combination of graphs.

\subsection{Double triangle reduction}
The double-triangle reduction is defined as follows.
\begin{figure}[ht]
\epsfig{file=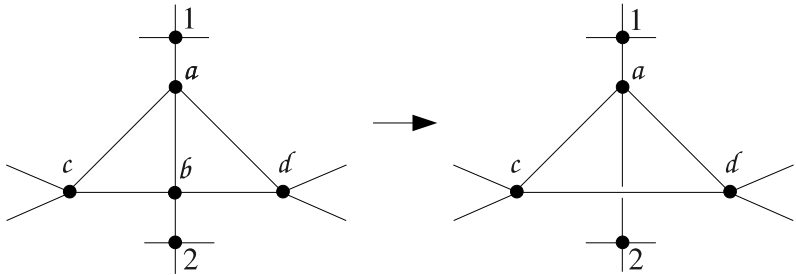,width=\textwidth}
\caption{Double triangle reduction: replace a joint vertex of two attached triangles by a crossing.}
\end{figure}

\begin{defn}\label{def2}
Assume a  graph $\Gamma$ has an edge $ab$ that is the common edge of (exactly) two triangles $(abc)$ and $(abd)$, $c\neq d$.
The  \emph{double triangle reduced graph} is defined to be  the graph in which the vertex $b$ is replaced by a crossing with edge $cd$ as depicted in Fig.\ 3.
\end{defn}
A double triangle reduced graph is completed primitive if and only if the original graph is completed primitive (the non-trivial direction of this statement is Prop.\ 2.19 in \cite{CENSUS}).
Let us define
$$\mathcal{G}_1 = \mathcal{G}_0 / \sim_{dt}$$
where $\sim_{dt}$ is the equivalence relation generated by $\Gamma_1 \sim_{dt} \Gamma_2$ if $\Gamma_1$ and $\Gamma_2$ are linked by a double triangle reduction.

Note that double triangle reduction does  \emph{not} preserve the period, but does respect the $c_2$-invariant, as we shall see below.
We define the \emph{family} of a completed primitive graph $\Gamma$ to be its equivalence class in $\mathcal{G}_1$.

\subsection{The ancestor of a family}
Every family has a unique smallest member in $\mathcal{G}_0$, which we call its ancestor.
\begin{defn}\label{def4}
A disjoint union of completed primitive graphs is an ancestor if none of its components can be reduced by the product  or double triangle.
An ancestor is prime if it is connected.
\end{defn}
\begin{thm}\label{thm2}
Every family has a unique ancestor.
\end{thm}
\begin{proof}
Prop.\ 2.21 and Prop.\ 2.22 of \cite{CENSUS}.
\end{proof}
The graph $O_3$ is in the family of the prime ancestor $K_5$.

\begin{remark}  By combining the results of \cite{BD} and \cite{BY}, we deduce that the piece of maximal (generic) Hodge-theoretic weight of the  cohomology of the graph hypersurface is invariant under double 
triangle reduction. In particular, a graph has `weight drop' in the sense of \cite{BY} if and only if its double-triangle reductions do also.
\end{remark}

\section{Counting points over finite fields}
Let $q=p^n$, where $p$ is a prime number, and let $\FF_q$ denote the finite field with $q$ elements.
 If $G$ is a graph  with $N_G$ edges, let 
\begin{equation}\label{9}
[X_G]_q=|\{x\in\FF_q^{N_G}:\Psi_G(x)=0\}| 
\end{equation}
denote the number of points of the affine graph hypersurface $X_G(\FF_q)$. 
If $G$ has at least three vertices then
$$0\equiv [X_G]_q \mod q^2$$ 
by    (\cite{BS}, Prop.-Def.\ 18).
This motivates the following definition: 
\begin{defn}\label{def5}
For a graph $G$ with at least three vertices, the $c_2$-invariant is the map which associates to every prime power $q$  the following element of $\ZZ / q \ZZ $:
\begin{equation}\label{10}
c_2(G)_q\equiv [X_G]_q/q^2\mod q.
\end{equation}
\end{defn}
The $c_2$-invariant should be thought of as a combinatorial version of the period of a graph, and carries the salient qualitative information
about the arithmetic nature of the period (see Conj.\ \ref{con2}). 

\subsection{Properties of the $c_2$-invariant}
Because graphs with isomorphic completions are proved to have the same period, Conj.\ \ref{con2} implies that the $c_2$-invariant only depends on the completion class of a graph.

\begin{con}\label{con4} (\cite{BS}, Conjectures\ 4 and 35)
The $c_2$-invariants of primitive-divergent graphs which have isomorphic completions are equal modulo $q$.
\end{con}

By proposition \ref{propindepofv}, Conj.\ \ref{con4} follows from Conj.\ \ref{con2}.
Conjecture \ref{con4} has been verified for many families, but although it seems much more tractable than Conj.\ \ref{con2} it is still unproved at present.
Nonetheless, we shall assume  Conj.\ \ref{con4} to be true throughout the remainder of this paper. Thus, we shall assume that the $c_2$-invariant lifts to the completed
primitive graph $\Gamma$ and we will use the notation
\begin{equation}\label{11}
c_2[\Gamma]_q\equiv c_2(\Gamma-v)_q\mod q
\end{equation}
to avoid any possible confusion with $c_2(\Gamma)_q$ (which we shall never consider
for completed graphs in this article). 
\begin{remark} Using the techniques of \cite{BSY, BY} and performing double-triangle reductions `at infinity', one can presumably show that the previous conjecture is true for a family if and only if it is true for the prime ancestor
of that family.
\end{remark}

Every completed primitive graph with at least 5 vertices
has a $c_2$-invariant and the first two examples are
\begin{equation}\label{12}
c_2[K_5]_q\equiv c_2[O_3]_q\equiv -1\mod q.
\end{equation}

If $G$ is uncompleted then in general $c_2(G)_q$ vanishes mod $q$ for graphs $G$ with vertex-connectivity $\leq 2$ (Prop.\ 31 in \cite{BS}).
Reducible completed primitive graphs have vertex-connectivity 3.
Hence the removal of one of the 3 connecting vertices provides a primitive graph with vanishing $c_2$-invariant.
\begin{prop}[assuming Conj.\ \ref{con4}]\label{prop1}
The $c_2$-invariants of reducible completed primitive graphs vanish modulo $q$.
\end{prop}
The $c_2$-invariant is invariant under double triangle reductions.
\begin{prop}[assuming Conj.\ \ref{con4}]\label{prop2}
If a completed primitive graph $\Gamma_1$ can be double triangle reduced to $\Gamma_2$ then $c_2[\Gamma_1]_q\equiv c_2[\Gamma_2]_q\mod q$.
\end{prop}
\begin{proof}
Because of (\ref{12}) we can assume that $\Gamma_1$ has at least 7 vertices. Hence $\Gamma_1$ has a vertex that is not involved in the double triangle reduction. We remove this vertex to obtain
a primitive-divergent graph $G$ which can be double triangle reduced. By Cor.\ 34 in \cite{BS} the $c_2$-invariant is unaffected by the double triangle reduction. Upon completion of the
reduced graph we obtain $\Gamma_2$.
\end{proof}
These  two propositions prove that the $c_2$-invariant factors through $\mathcal{G}_1$.
\begin{thm}[assuming Conj.\ \ref{con4}]\label{thm3}
All completed primitive graphs in a family have the same $c_2$-invariant. The $c_2$-invariant vanishes modulo $q$ if the ancestor of the family is not prime.
\end{thm}
Thus, assuming Conj.\ \ref{con4}, the number of $c_2$-invariants is at most the number of prime ancestors at a given loop order.
\vskip5mm

\begin{center}
\begin{tabular}{ccc}
loop order&comp.\ prim.\ graphs&prime ancestors\\\hline
3&1&1\\
4&1&0\\
5&2&0\\
6&5&1\\
7&14&4\\
8&49&10\\
9&227&37\\
10&1~354&231
\end{tabular} 
\end{center}
\vskip5mm

\noindent Table 2: The number of completed primitive graphs and prime ancestors up to 10 loops.
\vskip5mm
Table 2 shows that there are  a total of 284 prime ancestors up to loop order 10. These give  a complete list of $c_2$-invariants from $\phi^4$ theory up to loop order 10 (assuming 
conjecture  \ref{con4}). 

\subsection{Constant $c_2$-invariants} There are various families of (uncompleted) graphs for which the $c_2$-invariant is known.
In \cite{BS} it was proven that any graph of vertex-width at most 3 has a polynomial point-count. In particular, 

\begin{thm} \cite{BS} Let $G$ be a graph of vertex-width $\leq 3$. Then the $c_2$-invariant is constant:
$c_2(G)_q \equiv c \mod q$ for some constant $c\in \ZZ$.
\end{thm}
In \cite{BSY} it was shown that any graph in $\phi^4$ theory which is not primitive, i.e.,  containing a non-trivial subdivergence, has vanishing
$c_2$-invariant.

\begin{thm} Let $G$ be a log-divergent graph in $\phi^4$ theory. Then if $G$ has a strict subgraph $\gamma$ such that $2 h_{\gamma} \geq N_{\gamma} $,
it satisfies
$c_2(G)_q \equiv 0 \mod q$.
\end{thm}
Due to renormalization, and more specifically, the addition of counter-terms,  the connection between the $c_2$-invariant and the period is more complicated in the subdivergent case.
Even in `symmetric' cases where one can associate canonical periods to graphs with subdivergences it is not proved that their periods have weight drop.
However, examples (see e.g.\ Eq.\ (5.12) in \cite{Erik}) and some indirect arguments suggest that suitably defined periods of subdivergent graphs should have a double weight drop.

There are several other combinatorial criteria for a graph to have vanishing $c_2$-invariant which are discussed in \cite{BY}.

\subsection{Modularity} We are mainly interested in graphs whose 
 $c_2$-invariant   is congruent modulo $p$ to  the Fourier coefficients of a modular form.
\begin{defn}\label{def6}
A completed primitive graph $\Gamma$ is modular if there exists 
 a normalized Hecke eigenform $f$ for a congruence subgroup of $S\hspace{-1pt}L_2(\ZZ)$, possibly with a non-trivial Dirichlet-character, with an integral Fourier expansion
\begin{equation}\label{15}
f(\tau)=\sum_{k=0}^\infty a_kq^k,\quad q=\exp(2\pi{\rm i}\tau),\quad a_k\in\ZZ,
\end{equation}
such that the $c_2$-invariant modulo $p$ satisfies
\begin{equation}\label{16}
c_2[\Gamma]_p\equiv -a_p\mod p
\end{equation}
for all primes $p$. \end{defn}

For simplicity, we only consider point counts over fields with prime numbers of elements $p$, rather than the more general case of prime powers $q$, because
the latter can involve quadratic residue symbols which vanish modulo $p$ (this is the case for the graph $P_{8,37}$, whose modularity was proved in \cite{BS}).
Note that it is expected, but not known in general,  whether  a modular form of the type considered in Def.\ \ref{def6} is uniquely defined by its coefficients $a_p$ modulo $p$ for all primes $p$.

\section{Denominator reduction}\label{DR}
  Although the $c_2$-invariant appears to be  the most complicated part of the graph hypersurface point-count, the method of denominator reduction 
provides a tool to access precisely this part of the point-count \cite{BS}.

\subsection{Matrix representation}
We recall some basic results from \cite{BROWN}. We will use the following matrix representation for the graph polynomial. 

Choose an orientation of the edges of $G$, and for every edge $e$ and vertex $v$ of $G$, define the incidence matrix:
\begin{equation}
(\mathcal{E}_G)_{e,v} = \left\{
\begin{array}{rl}
1, & \hbox{if the edge } e \hbox{ begins at } v \hbox{ and does not end at } v,\\
-1, & \hbox{if the edge } e  \hbox{ ends at } v \hbox{ and does not begin at } v,\\
0, & \hbox{otherwise}.
\end{array}
\right.
\end{equation}
Let $A$ be the diagonal matrix with entries $x_e$, for $e \in E(G)$, the set of edges of $G$, and set
\begin{equation}
\widetilde{M}_G=\left(
\begin{array}{c|c}
A  & \mathcal{E}_G  \\\hline
{-}\mathcal{E}_G^T&  0  \\
\end{array}
\right)
\end{equation}
where the first $e_G$ rows and columns are indexed by $E(G)$,
and the remaining $v_G$ rows and columns are indexed by the set of vertices of $G$, in some order.
The matrix $\widetilde{M}_G$ has corank $\geq 1$. Choose any vertex of $G$ and let $M_G$ denote the square $(N_G+V_G-1)\times (N_G+V_G-1)$ matrix obtained from it by deleting the row and column indexed by this vertex.

It follows from the matrix-tree theorem that the graph polynomial satisfies 
\begin{equation}\label{17}
\Psi_G=\det (M_G).
\end{equation}

\subsection{The five-invariant}
Let $I,J,K$ be subsets of the set of edges of $G$ which satisfy $|I|=|J|$. Let $M_G(I,J)_K$ denote the matrix obtained from $M_G$ by removing the rows (resp. columns)
indexed by the set $I$ (resp.\ $J$) and setting $x_e=0$ for all $e\in K$. Let 
\begin{equation} \label{Dogsondefn} \Psi_{G,K}^{I,J}=\det M_G(I,J)_K\ . \end{equation}
Now let $i,j,k,l,m$ denote any five distinct edges in a graph $G$. The five-invariant of these edges, denoted ${}^5 \Psi_G(i,j,k,l,m)$ is defined up to a sign, and is given by the determinant
\begin{equation}\label{18}
{}^5\Psi_G(i,j,k,l,m) = \pm \det \left(
\begin{array}{cc}
 \Psi^{ij,kl}_{G,m} &   \Psi^{ik,jl}_{G,m}   \\
 \Psi^{ijm,klm}_{G} &   \Psi^{ikm,jlm}_{G}
\end{array}
\right).
\end{equation}
It  is well-defined, i.e., permuting the five indices $i,j,k,l,m$ only modifies the right-hand side by a sign.
In general, the 5-invariant is irreducible of degree 2 in each variable $x_e$. However, in many cases it factorizes into a product of polynomials each of which is
linear in a variable $x_6$. In this case denominator reduction allows us to further eliminate variables by taking resultants.

\subsection{Reduction algorithm}
Given a graph $G$ and an ordering $e_1,$ $\ldots,$ $e_{N_G}$ on its edges, we can extract a sequence of higher invariants (each defined up to a sign) as follows.
Define $D^5_G(e_1,$ $\ldots,$ $e_5) = {}^5 \Psi_G(e_1,$ $\ldots,$ $e_5)$. Let  $n\geq 5$ and suppose that we have defined $D^n_G(e_1,$ $\ldots,$ $e_n)$. Suppose furthermore that
$D^n_G(e_1,$ $\ldots,$ $e_n)$ factorizes into a product of linear factors in $x_{n+1}$, i.e., it is of the form 
$(ax_{n+1}+b) (cx_{n+1}+d)$. Then we define
\begin{equation}\label{19}
{D}^{n+1}_G(e_1,\ldots, e_{n+1}) =\pm( ad-bc),
\end{equation}
to be the resultant of the two factors of ${D}^n_G(e_1,\ldots, e_n)$ with respect to $x_{n+1}$.
A graph $G$ for which the polynomials ${D}^n_G(e_1,\ldots, e_{n})$ can be defined for all $n$ is called \emph{denominator-reducible}.

One can prove, as for the 5-invariant, that ${D}^n_G(e_1,\ldots, e_n)$ does not depend on the order of reduction of the variables,
although it may happen that the intermediate terms  ${D}^k_G(e_{i_1},\ldots, e_{i_k})$ may factorize for some choices of orderings and not others.

Denominator reduction is connected to the $c_2$-invariant by the following theorem (Thm.\ 29 in \cite{BS} or Eq.\ (2.33) in \cite{SchnetzFq}):
\begin{thm} \label{thm4}
Let $G$ be a connected graph with $N_G\geq5$ edges and $h_G\leq N_G/2$ independent cycles. Suppose that $D^n_G(e_1,\ldots, e_n)$ is the result of
the denominator reduction after $n<N_G$ steps. Then
\begin{equation}\label{20}
c_2(G)_q \equiv  (-1)^n [D^n_G(e_1,\ldots, e_n)]_q \mod q.
\end{equation}
\end{thm}

For general graphs above a certain loop order and any ordering on their edges, there will come a point where ${D}^n_G(e_1,\ldots, e_n)$ is irreducible (typically for $n=5$).
Thus the generic graph is not denominator reducible.
The $c_2$-invariant of a graph $G$ vanishes if $h_G<N_G/2$ (Cor.\ 2.10 in \cite{SchnetzFq}).

\section{Results}\label{results}

Up to loop order 10 we have a total of 284 prime ancestors (see Tab.\ 2). With the first 6 primes (whose product is 30030) we can distinguish 145 $c_2$-invariants, which are listed in
Tab.\ 7 and Tab.\ 8 at the end of this section. For the 16 prime ancestors up to loop order 8 we have determined the $c_2$-invariant for at least the first 26 primes;
for the 37 graphs at loop order 9 the $c_2$-invariant is known for at least the first 12 primes (see Tab.\ 6). At loop order 10 we evaluated the $c_2$-invariant for many graphs for $p=17$
and beyond without resolving any more $c_2$-invariants.
Therefore we expect that the total number of $c_2$-invariants up to loop order 10 does not significantly exceed 145. Although a few of the many identities among $c_2$-invariants
are explained by twist and Fourier identities (see Sect.\ \ref{GN}) most  do not seem to follow from any known identities.

We found three types of $c_2$-invariants: quasi-constant, modular, and unidentified sequences.
\subsection{Quasi-constant graphs}

\begin{figure}[ht]
\epsfig{file=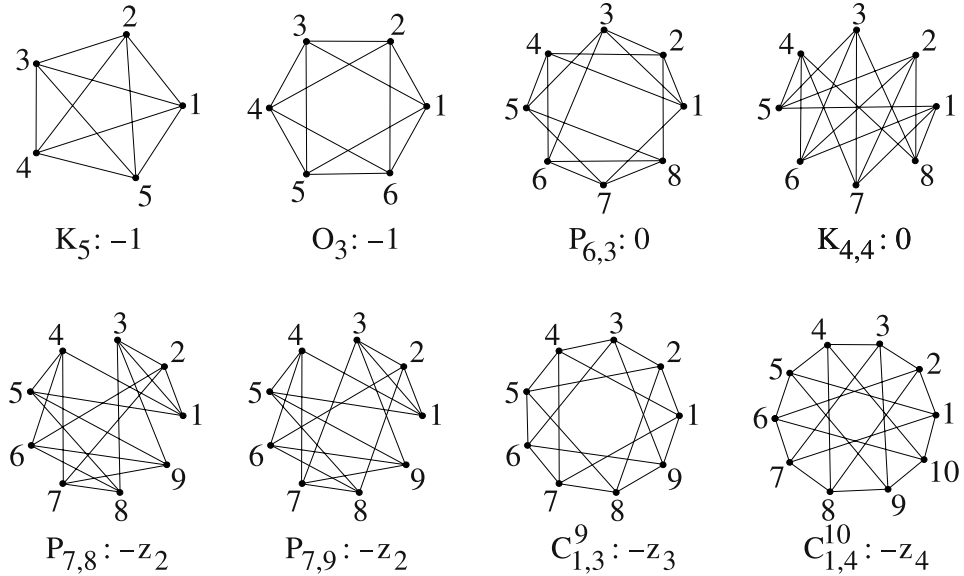,width=\textwidth}
\caption{Quasi-constant $\phi^4$-graphs with their $c_2$-in\-vari\-ants. The ancestor of $O_3$ is $K_5$, the ancestor of $P_{6,3}$ is $K_5^2$. The graphs
$K_5$, $K_{4,4}$, $P_{7,8}$, $P_{7,9}$, $C_{1,3}^9$, $C_{1,4}^{10}$ are the smallest examples of ancestors with their $c_2$-invariant.
The names of the graphs are taken from \cite{CENSUS} and $z_\bullet$ is defined in (\ref{21}).}
\end{figure}

\begin{defn}\label{def7}
A completed primitive graph $G$ is quasi-constant (quasi $c$)
if there exists an $m\in\NN_+$ such that $c_2(G)_{p^{mn}}\equiv c\mod p^{mn}$ for almost all primes $p$ and all $n\in\NN_+$.
\end{defn}
For $k\geq 2$, consider the following residue symbol, denoted $z_k$, which depends on the number of $k$-th roots of unity in $\FF_q$,
\begin{equation}\label{21}
z_k(q)=\left\{
\begin{array}{rl}
1&\hbox{if }\#\{x\in\FF_q :x^k=1\}=k,\\
0&\hbox{if gcd}(k,q)>1,\\
-1&\hbox{otherwise.}
\end{array}\right.
\end{equation}
It is a quasi-constant (quasi $1$). There is an abundance of quasi $-1$ $c_2$-invariants, which  is partly explained by denominator reduction, Thm.\ \ref{thm4}.
\begin{lem}
Let $G$ be a graph with $N_G=2h_G$, which  is denominator reducible up to the penultimate stage: i.e.\ there exists a sequence on the edges of $G$ such that $D^{N_G-2}_G$ is defined.
Then $c_2(G)_q$ is either 0 or quasi $-1$.
\end{lem}
\begin{proof}
Theorem \ref{thm4} is proved in \cite{BS} by iteratively taking resultants which (if non-zero) factorize into terms which are linear in the next variable. At each step the degree of the
resultant goes down by one. Because we start with $D^5_G$ which has degree $N_G-5$ we observe that $D^n_G$ is either zero or of degree $N_G-n$.
If $D^{N_G-2}_G=0$ then the $c_2$-invariant vanishes mod $q$.
Otherwise $D^{N_G-2}_G$ is a homogeneous quadratic polynomial in two variables with integer coefficients and defines a hypersurface $V$ of degree two in affine space $\mathbb{A}^2$.
Let $V_q$ denote its reduction in $\FF_q^2$.
If $D^{N_G-2}_G$ is a perfect square then $V$ is a double line and $V_q$ is either isomorphic to $\FF_q$ or---if $D^{N_G-2}_G$ vanishes mod $q$---to $\FF_q^2$.
In both cases the $c_2$-invariant vanishes  mod $q$.
Now we assume that $D^{N_G-2}_G$ is not a perfect square. Let $\Delta\in\ZZ^\times$ be the discriminant of $D^{N_G-2}_G$.
For every prime $p$ not dividing $\Delta$
and every $n\in\NN_+$,  $\Delta$ is a non-zero square in $\FF_{p^{2n}}$ and  $V_{p^{2n}}$ is isomorphic to a union of two lines meeting at a point.
Thus $|V_{p^{2n}}| = |\FF_{p^{2n}}| + |\FF_{p^{2n}}| - |pt| = 2p^{2n}-1$ which makes the $c_2$-invariant quasi $-1$.
\end{proof}

The graphs satisfying the conditions of the  lemma have the property that the maximal weight-graded piece of the graph cohomology is one-dimensional
and spanned by the Feynman differential form (see \cite{BD}).
\vskip1ex

For $\phi^4$ graphs we observe that  up to loop order 6 all  $c_2$-invariants are  in fact constant,  equal to $0$ or $-1$.
At loop order 7 all new $c_2$-invariants are quasi-constant, equal to
$-z_2$ or  $-z_3$. The only other quasi-constant $c_2$-invariant first appears at loop order 8, and equals $-z_4$. No new quasi-constant
$c_2$-invariants were found at loop orders $9$ and $10$, which leads us to conjecture that this list is complete.
Note that the cases $z_2,z_3, z_4$ become quasi-constant after excluding  the prime $2$, or by adjoining 3rd and 4th roots of unity, respectively.
These are precisely the extensions of $\QQ$ by roots of unity which are trivial or quadratic. Thus the conjecture is partly supported by the previous lemma, which only involves quadratic extensions.

Note, however, that for many quasi-constant graphs an edge ordering that leads to $D^{N_G-2}_G$ is not known.
Furthermore, there is not a single graph with  $c_2=-z_3$ or $-z_4$ whose period is known.
What we do have is the conjectured period of two completed primitive graphs at loop order 7 (graphs $P_{7,8}$ and $P_{7,9}$ in \cite{CENSUS}) with $c_2$-invariant $-z_2$.
They are weight 11 multiple zeta values, namely \cite{B3}
\begin{eqnarray}
P_{7,8}&=&\frac{22383}{20}\zeta(11)-\frac{4572}{5}[\zeta(3)\zeta(3,5)-\zeta(3,5,3)]-700\zeta(3)^2\zeta(5)\nonumber\\
&&\quad+\,1792\zeta(3)\left[\frac{27}{80}\zeta(3,5)+\frac{45}{64}\zeta(5)\zeta(3)-\frac{261}{320}\zeta(8)\right]\hbox{ and}
\end{eqnarray}
\begin{eqnarray}
P_{7,9}&=&\frac{92943}{160}\zeta(11)-\frac{3381}{20}[\zeta(3)\zeta(3,5)-\zeta(3,5,3)]-\frac{1155}{4}\zeta(3)^2\zeta(5)\nonumber\\
&&\quad+\,896\zeta(3)\left[\frac{27}{80}\zeta(3,5)+\frac{45}{64}\zeta(5)\zeta(3)-\frac{261}{320}\zeta(8)\right].
\end{eqnarray}

To summarize the results for quasi-constant graphs, the 10 loop data is consistent with the following, rather surprising, conjecture.
\begin{con}\label{con5}
If a completed primitive graph $\Gamma$ is quasi-constant then its $c_2$-invariant is 0, $-1$,
$-z_2$, $-z_3$, or $-z_4$ [see (\ref{21})].

If $c_2[\Gamma]_q\equiv-1\mod q$ then the ancestor of $\Gamma$ is $K_5$.  
\end{con}
The second part of the conjecture broadly states that primitive $\phi^4$ graphs fall into three categories: the weight-drop graphs with $c_2$-invariant 0, complicated graphs which have non-polynomial point counts (whose prime ancestors have at least 7 loops), and a single family of graphs whose ancestor is $K_5$. The last category has many distinguished combinatorial properties. 
\vskip5mm

\begin{tabular}{r|r||l}
weight&level&modular form\\\hline
2&11&$\eta(z)^2\eta(11z)^2$\\[2pt]
2&14&$\eta(z)\eta(2z)\eta(7z)\eta(14z)$\\[2pt]
2&15&$\eta(z)\eta(3z)\eta(5z)\eta(15z)$\\[2pt]
3&7&$\eta(z)^3\eta(7z)^3$\\[2pt]
3&8&$\eta(z)^2\eta(2z)\eta(4z)\eta(8z)^2$\\[2pt]
3&12&$\eta(2z)^3\eta(6z)^3$\\[2pt]
4&5&$\eta(z)^4\eta(5z)^4$\\[2pt]
4&6&$\eta(z)^2\eta(2z)^2\eta(3z)^2\eta(6z)^2$\\[2pt]
5&4&$\eta(z)^4\eta(2z)^2\eta(4z)^4$\\[2pt]
6&3&$\eta(z)^6\eta(3z)^6$\\[2pt]
6&4&$\eta(2z)^{12}$\\[2pt]
8&2&$\eta(z)^8\eta(2z)^8$
\end{tabular}
\vskip5mm

\noindent Table 3: Newforms $f(z)$ from Tab.\ 1 which are expressible as products of the Dedekind $\eta$-function.
\vskip5mm

\subsection{Modular graphs}
\begin{figure}[ht]
\epsfig{file=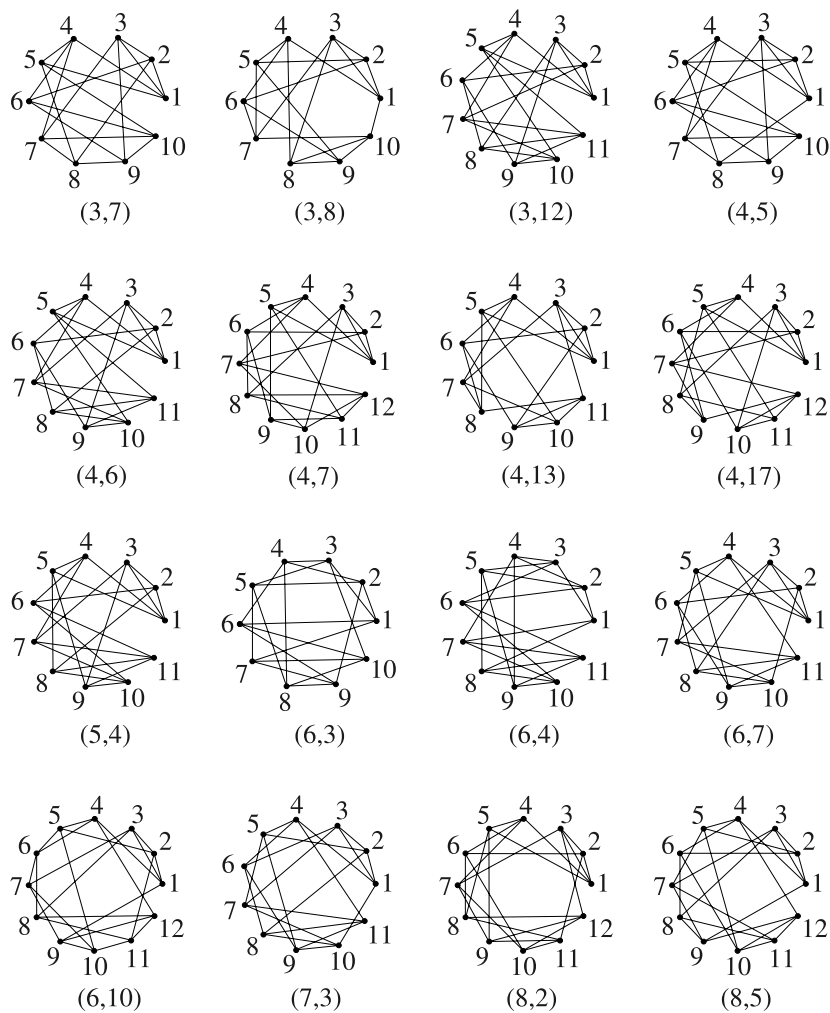,width=\textwidth}
\caption{The 16 modular $\phi^4$ graphs up to loop order 10. The numbers beneath the figures refer to the (weight, level) of the corresponding modular form.}
\end{figure}

The first modular graphs (see Defn.\ \ref{def6}) appear at loop order 8. In fact all four non quasi-constant graphs at loop order 8 are modular with respect to the newforms
with weight and level equal to $(3, 7)$, $(3, 8)$, $(4, 5)$ and $(6, 3)$. The graphs are depicted in Fig.\ 5.
At loop order 9 we are able to identify another 7 modular graphs whereas at loop order 10 we find only 5 new modular sequences. All modular graphs were found to have very small levels ($\leq 17$, see Tab.\ 1). A search for much higher levels did not provide any new fits. Nine out of the 16 modular graphs 
correspond to products of the Dedekind $\eta$-function (see Tab.\ 3 and \cite{etamod}).
All modular $c_2$-invariants are confirmed for at least the first 11 primes (whose product is
$\sim2\cdot 10^{11}$). 

An unexpected observation is the absence of weight 2.
\begin{con}\label{con6}
If a completed primitive graph is modular with respect to the modular form $f$ then the weight of $f$ is $\geq3$.
\end{con}
In particular, $\phi^4$ point-counts are conjecturally not of general type (compare the main theorem of  \cite{BB}).
In support of this conjecture,  we find by computation that any counter-example $\Gamma$ with $\leq 10$ loops, would have to correspond to a modular
form of weight $2$ and  level $\geq 1200$.

\subsection{Unidentified graphs}

\begin{figure}[ht]
\epsfig{file=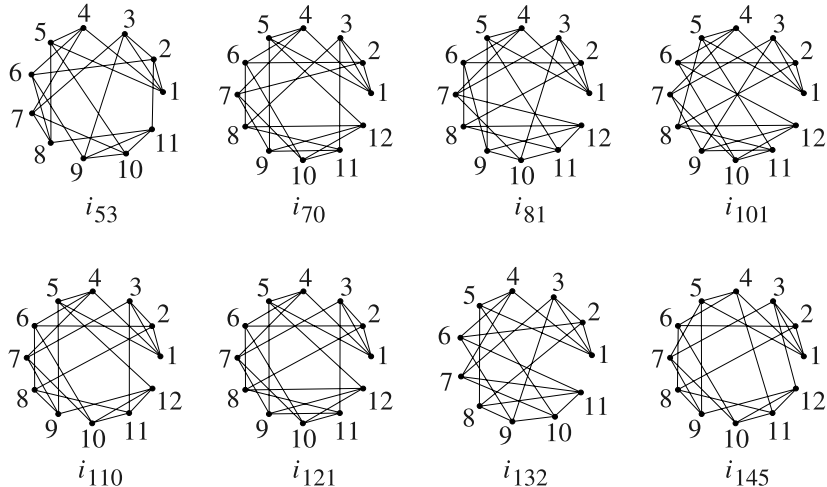,width=\textwidth}
\caption{Eight graphs with readily accessible but as-yet unidentified $c_2$-invariants. Their point-counts mod $p$ for the first 50 primes are listed in Tab.\ 5.}
\end{figure}

The first unidentified $c_2$-invariants appear at loop order 9 where 10 sequences cannot be associated to quasi-constants or modular forms. At loop order 10, there are  114 as  yet unidentified sequences
out of 231 prime ancestors. Because modular $c_2$-invariants appear to have low levels it is reasonable to expect that
the unidentified sequences are not modular with respect to congruence subgroups of $S\hspace{-1pt}L_2(\ZZ)$.

The most accessible unidentified sequences are those for which the denominator reduction algorithm continues for as long as possible, and  provides  homogeneous  polynomials of degree 7 in  7  variables.
In Tab.\ 5 we list $-c_2$ for the first 50 primes for the eight graphs which fall into this class. The graphs are depicted in Fig.\ 6.

Using the rescaling technique presented in \cite{SchnetzFq} we were able to reduce the $c_2$-invariant of the unidentified sequence $i_{101}$ one step further to   a projective 4-fold of degree 6.
It is given explicitly by the homogeneous polynomial
\begin{equation}\label{23}
AB  +x_2^2(x_4+x_5)^2 x_5x_6 + C x^2_5x_6 
\end{equation}
where
\begin{eqnarray}
A&=& (x_2+x_5) (x_3+x_6)x_4+x_2x_3x_5+x_2x_3x_6+x_2x_5x_6+ x_3x_5x_6\nonumber  \\
B&=&   x_1 (x_1x_3+x_3x_4+x_3x_6-x_2x_4-x_2x_5-x_5x_6)  \nonumber \\
C& =& (x_4+x_6) (x_2x_4+x_2x_5-x_2x_3-x_3x_4-x_3x_6-x_4x_6 ).\nonumber 
\end{eqnarray}
The $c_2$-invariant for this sequence is listed as $i_{101}$ in Tab.\ 5 for the first 100 primes. Remarkably, $i_{101}(p)$ is a square in $\FF_p$ if and only if $p\not\equiv-1$ mod 12.
\vskip5mm

\subsection{Non-$\phi^4$ graphs}
\begin{figure}[ht]
\epsfig{file=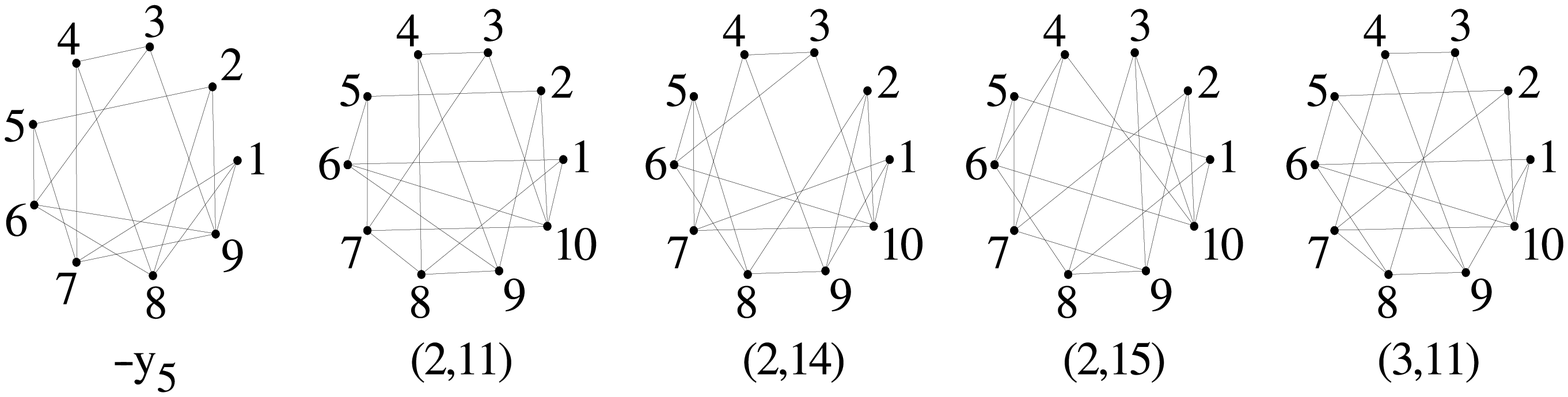,width=\textwidth}
\caption{Non-$\phi^4$ graphs which are quasi-constant ($-y_5$, see Eq.\ (\ref{24})) or modular with subscript (weight, level). Note that all graphs have a non-$\phi^4$ (5-valent) vertex.}
\end{figure}

Since denominator reduction Thm.\ \ref{thm4} applies to any connected graph $G$ satisfying $N_G=2h_G$,  we can also study the $c_2$-invariants of graphs which are not necessarily in $\phi^4$ theory
(i.e.\ which have a vertex of degree $\geq 5$). The condition $N_G= 2h_G$ means that  these graphs are superficially log divergent.
If, furthermore, they are free of subdivergences  (\ref{defnprimdiv1}),  they are primitive divergent and have a convergent  period (\ref{2}).
Note that there is no completion conjecture for the $c_2$-invariants of non-$\phi^4$ graphs.

Because the presence of a vertex of degree $\leq2$ causes the $c_2$-invariant to vanish by Lemma 17 (10) in \cite{BS} we may restrict ourselves to  graphs with minimum vertex degree $\geq3$.
All 23 graphs of this type up to loop order 6 have $\leq12$ edges and hence constant $c_2$-invariant by  the work of Stembridge  \cite{STEM}.
We find that 11 of the $c_2$-invariants are $-1$ whereas all others are 0.

At loop order 7 we have 133 graphs. There are no new  $c_2$-invariants in the sense that they already occur in $\phi^4$.

Among the 1352 graphs at 8 loops there is the first non-$\phi^4$ $c_2$-invariant. Its graph $G_5$ is drawn in Fig.\ 7.  It satisfies $c_2(G_5)_q\equiv -y_5\mod q$, where
\begin{equation}\label{24}
y_5=\#\{x\in\FF_q:x^2+x-1=0\}-1.
\end{equation}
and has been verified for the first 50 primes.

At loop order 9, we find modular forms of type (weight, level) equal to $(2,11)$, $(2,14)$, $(2,15)$ (see Tab.\ 1) which do not occur in $\phi^4$ theory.
We also obtain a
newform of  weight 3 and level 11 which is absent in $\phi^4$-theory up to loop order 10.
Moreover, we find the two modular forms of weight 4, level 7 and weight 4, level 17 which occur in $\phi^4$ at 10 loops. 
These modular $c_2$-invariants  are confirmed for at least 11 primes. 
Finally, the $c_2$-invariant of 15 graphs give 7 new unidentified sequences which we do not list here.

\section{Methods}

We used B.D. McKay's {\it nauty} to generate graphs \cite{NAU}, and Sage  \cite{Sage} to generate  modular forms (supplemented by comprehensive tables in weights 2, 4, 6 which can be found in \cite{Meyer}).
The first reduction to prime ancestors reduces the total number of graphs to be considered from 1354 to 231 in the case of $\phi^4$ theory at 10 loops. This assumes the completion conjecture 
\ref{con4}.  

The denominator reduction  then reduces the $c_2$-invariant  of each (uncompleted) graph to the point-count of a hypersurface of smaller degree, as follows.
Primitive-divergent $\phi^4$ graphs with
at least five vertices have at least two 3-valent vertices with no common edge joining them.
If we choose the first 6 edges $e_1,\ldots,e_6$ to contain these two three-valent vertices
then one can show \cite{BROWN} that $D^5$ and $D^6$ necessarily factorize. Therefore for $\phi^4$ graphs we can always apply denominator reduction
twice to reduce the number of variables by 7.  Since  $D^7$ is homogeneous, we can eliminate  a further  variable with the effect that at 10 loops we have to count the zeros of a polynomial in at most 12 variables.
A further variable can  effectively be eliminated using the fact that $D^7$ is quadratic in every variable and by computing tables of quadratic residues modulo $p$. 
Altogether for the prime 13 we need at 10 loops a maximum of $13^{11}\approx 1.8\cdot 10^{12}$ evaluations of typically $\sim$30~000 monomials to determine
$c_2(G)_{13}$. For only 10 out of 231 prime 10 loop ancestors we did not manage to find more than the minimum number of reductions. The point-counts for these 10 graphs with primes
$p=11$ and $p=13$ were performed at the Erlanger RRZE Cluster using a total computation time of 18 core years. The results are $i_{20}$ (twice), $i_{24}$, $i_{34}$, $i_{49}$, $i_{60}$, $i_{64}$,
$i_{75}$, $i_{78}$, $i_{85}$ in Tab.\ 8. A further 16 core years were invested in the sequence $i_{49}$ at prime 17 to rule out an accidental coincidence with a weight 2 modular form of level 624 for the first 6 primes.
Additionally the point-count of the weight 8 level 5 modular graph for the primes 29 and 31 consumed 1 core year on the RRZE Cluster.
For the 9 loop graph $P_{9,190}$ (see Tab.\ 6) we evaluated the point counts of the primes 29, 31, and 37 on the RRZE Cluster in altogether 3 core years. All other computations were performed
using an office PC.
\newpage

\section{Tables}

\begin{tabular}{r||r|r|r|r|r|r|rrr||r|rrr||r}
$-c_2$&\multicolumn{7}{l}{loop order $\ell=$}&&$-c_2$&\multicolumn{2}{l}{$\ell=$}&&$-c_2$&$\ell=$\\
&3&4&5&6&7&8&9&&&8&9&&&9\\\hline
1&1&1&2&7&42&393&4698&&(3,12)&0&8&&$i_{53}$&11\\
$z_2$&0&0&0&0&5&49&763&&(4,5)&3&73&&$i_{81}$&6\\
$z_3$&0&0&0&0&1&19&375&&(4,6)&0&36&&$i_{87}$&5\\
$z_4$&0&0&0&0&0&1&15&&(4,7)&0&5&&$i_{104}$&1\\
$y_5$&0&0&0&0&0&1&34&&(4,13)&0&8&&$i_{106}$&6\\
(2,11)&0&0&0&0&0&0&6&&(4,17)&0&2&&$i_{110}$&10\\
(2,14)&0&0&0&0&0&0&4&&(5,4)&0&29&&$i_{121}$&6\\
(2,15)&0&0&0&0&0&0&2&&(6,3)&1&7&&$i_{125}$&6\\
(3,7)&0&0&0&0&0&6&133&&(6,4)&0&3&&$i_{132}$&11\\
(3,8)&0&0&0&0&0&4&81&&(6,7)&0&10&&$i_{141}$&6\\
(3,11)&0&0&0&0&0&0&3&&(7,3)&0&3&&$i_{\phi^{>4}}$&15
\end{tabular}
\vskip5mm

\noindent Table 4: All non-zero $c_2$-invariants of `log divergent' graphs (graphs for which the number of edges equals twice the number of loops) up to 9 loops.
The functions $z_\bullet$ are defined in Eq.\ (\ref{21}), $y_5$ is defined in Eq.\ (\ref{24}), (weight, level) refer to modular forms, and $i_\bullet$ refers to an unidentified sequence
listed in Tab.\ 8. $i_{\phi^{>4}}$ refers to unidentified sequences in non-$\phi^4$ graphs.
\vskip5mm

\begin{tabular}{r||r|r|r|r|r|r|r|r||r||r}
$p$&$i_{53}$&$i_{70}$&$i_{81}$&$i_{101}$&$i_{110}$&$i_{121}$&$i_{132}$&$i_{145}$&$p$&$i_{101}$\\\hline
2&0&0&0&1&1&1&1&1&233&232\\
3&1&1&2&0&0&1&2&2&239&179\\
5&1&4&2&1&4&3&1&4&241&196\\
7&4&2&2&4&6&3&4&6&251&183\\
11&6&1&6&6&0&10&1&9&257&135\\
13&6&0&5&9&0&8&0&10&263&94\\
17&14&7&14&16&9&14&4&8&269&268\\
19&5&1&10&9&3&1&11&1&271&242\\
23&15&1&1&17&6&14&6&13&277&210\\
29&11&10&27&28&24&3&23&17&281&261\\
31&21&30&16&1&28&9&2&22&283&9\\
37&23&5&12&30&4&24&13&20&293&240\\
41&1&28&6&8&17&19&12&4&307&289\\
43&33&31&15&6&36&10&13&21&311&51\\
47&33&14&0&11&36&11&6&8&313&108\\
53&47&8&48&13&46&19&44&14&317&113\\
59&25&25&11&44&16&8&29&30&331&43\\
61&40&5&17&42&46&26&14&23&337&100\\
67&60&22&10&33&24&11&35&54&347&97\\
71&53&59&30&65&29&33&15&27&349&273\\
73&46&55&36&64&22&14&2&67&353&334\\
79&25&68&50&38&72&35&43&14&359&335
\end{tabular}

\begin{tabular}{r||r|r|r|r|r|r|r|r||r||r}
$p$&$i_{53}$&$i_{70}$&$i_{81}$&$i_{101}$&$i_{110}$&$i_{121}$&$i_{132}$&$i_{145}$&$p$&$i_{101}$\\\hline
83&51&64&8&74&2&15&23&30&367&137\\
89&65&12&60&88&1&40&0&44&373&179\\
97&33&78&53&47&89&91&7&10&379&263\\
101&1&80&34&52&31&48&76&83&383&310\\
103&33&63&25&60&58&16&71&76&389&176\\
107&45&12&50&17&49&21&45&83&397&136\\
109&42&38&34&82&56&5&24&43&401&36\\
113&40&38&72&102&56&39&1&0&409&377\\
127&62&19&61&115&30&45&43&18&419&314\\
131&8&96&69&37&3&104&99&97&421&51\\
137&125&72&0&136&61&10&7&43&431&370\\
139&65&48&61&89&113&20&137&29&433&104\\
149&86&102&131&37&70&148&132&108&439&182\\
151&24&9&127&40&79&30&59&21&443&76\\
157&155&124&107&109&120&135&100&70&449&161\\
163&99&34&60&38&88&46&39&158&457&238\\
167&22&117&134&95&60&60&2&37&461&39\\
173&56&85&75&157&21&73&23&55&463&388\\
179&4&37&122&32&173&46&132&42&467&234\\
181&91&33&23&75&1&24&65&90&479&325\\
191&53&46&45&55&136&115&4&170&487&412\\
193&92&31&175&98&155&126&29&156&491&343\\
197&39&170&68&97&185&132&23&56&499&431\\
199&174&126&121&178&158&133&5&173&503&5\\
211&8&33&182&53&189&176&146&90&509&429\\
223&1&119&110&175&37&216&41&184&521&169\\
227&183&199&105&52&224&110&70&37&523&263\\
229&115&20&63&9&225&94&181&71&541&412
\end{tabular}
\vskip5mm

\noindent Table 5: The 8 most accessible unidentified sequences (see Fig.\ 6). We list $-c_2[\Gamma]_p$ for the first 50 primes and for the first 100 primes in case of $i_{101}$.
\vskip5mm

\begin{tabular}{r||r||r|r|r|r|r|r|r|r|r|r|r|r}
graph&$-c_2(p)$&\multicolumn{12}{l}{$p=$}\\
&&\hspace{4pt}2&\hspace{4pt}3&\hspace{4pt}5&\hspace{4pt}7&11&13&17&19&23&29&31&37\\\hline
$P_3$&1&1&1&1&1&1&1&1&1&1&1&1&1\\\hline
$P_{6,4}$&0&0&0&0&0&0&0&0&0&0&0&0&0\\\hline
$P_{7, 8}$&$z_2$&0&1&1&1&1&1&1&1&1&1&1&1\\ 
$P_{7, 9}$&$z_2$&0&1&1&1&1&1&1&1&1&1&1&1\\
$P_{7, 10}$&0&0&0&0&0&0&0&0&0&0&0&0&0\\
$P_{7, 11}$&$z_3$&1&0&4&1&10&1&16&1&22&28&1&1\\\hline
$P_{8, 32}$&0&0&0&0&0&0&0&0&0&0&0&0&0\\
$P_{8, 33}$&$z_3$&1&0&4&1&10&1&16&1&22&28&1&1\\
$P_{8, 34}$&0&0&0&0&0&0&0&0&0&0&0&0&0
\end{tabular}

\begin{tabular}{r||r||r|r|r|r|r|r|r|r|r|r|r|r}
graph&$-c_2(p)$&\multicolumn{12}{l}{$p=$}\\
&&\hspace{4pt}2&\hspace{4pt}3&\hspace{4pt}5&\hspace{4pt}7&11&13&17&19&23&29&31&37\\\hline
$P_{8, 35}$&$z_2$&0&1&1&1&1&1&1&1&1&1&1&1\\
$P_{8, 36}$&$z_3$&1&0&4&1&10&1&16&1&22&28&1&1\\
$P_{8, 37}$&(3, 7)&1&0&0&0&5&0&0&0&18&4&0&36\\
$P_{8, 38}$&(4, 5)&0&2&0&6&10&1&9&5&14&8&16&7\\
$P_{8, 39}$&(3, 8)&0&1&0&0&3&0&2&4&0&0&0&0\\
$P_{8, 40}$&$z_4$&0&2&1&6&10&1&1&18&22&1&30&1\\
$P_{8, 41}$&(6, 3)&0&0&1&2&8&1&15&14&11&27&29&32\\\hline
$P_{9, 154}$&$z_2$&0&1&1&1&1&1&1&1&1&1&1&1\\
$P_{9, 155}$&(3, 7)&1&0&0&0&5&0&0&0&18&4&0&36\\
$P_{9, 156}$&$z_2$&0&1&1&1&1&1&1&1&1&1&1&1\\
$P_{9, 157}$&(3, 12)&0&0&0&2&0&4&0&7&0&0&16&26\\
$P_{9, 158}$&(3, 8)&0&1&0&0&3&0&2&4&0&0&0&0\\ 
$P_{9, 159}$&$z_2$&0&1&1&1&1&1&1&1&1&1&1&1\\ 
$P_{9, 160}$&(3, 8)&0&1&0&0&3&0&2&4&0&0&0&0\\
$P_{9, 161}$&(5, 4)&0&0&1&0&0&9&16&0&0&24&0&16\\
$P_{9, 162}$&(4, 5)&0&2&0&6&10&1&9&5&14&8&16&7\\
$P_{9, 163}$&(4, 5)&0&2&0&6&10&1&9&5&14&8&16&7\\
$P_{9, 164}$&(3, 7)&1&0&0&0&5&0&0&0&18&4&0&36\\
$P_{9, 165}$&$i_{132}$&1&2&1&4&1&0&4&11&6&23&2&13\\
$P_{9, 166}$&(4, 6)&0&0&1&5&1&12&10&1&7&1&5&32\\
$P_{9, 167}$&(4, 13)&1&2&3&1&7&0&9&7&19&5&10&17\\
$P_{9, 168}$&(3, 8)&0&1&0&0&3&0&2&4&0&0&0&0\\
$P_{9, 169}$&$i_{53}$&0&1&1&4&6&6&14&5&15&11&21&23\\
$P_{9, 170}$&(5, 4)&0&0&1&0&0&9&16&0&0&24&0&16\\
$P_{9, 171}$&(4, 6)&0&0&1&5&1&12&10&1&7&1&5&32\\
$P_{9, 172}$&(4, 6)&0&0&1&5&1&12&10&1&7&1&5&32\\
$P_{9, 173}$&(6, 7)&0&1&4&0&1&3&12&16&7&25&13&36\\
$P_{9, 174}$&(4, 6)&0&0&1&5&1&12&10&1&7&1&5&32\\
$P_{9, 175}$&(4, 5)&0&2&0&6&10&1&9&5&14&8&16&7\\
$P_{9, 176}$&(6, 7)&0&1&4&0&1&3&12&16&7&25&13&36\\
$P_{9, 177}$&$i_{141}$&1&2&3&3&0&9&0&1&4&6&18&15\\
$P_{9, 178}$&$i_{125}$&1&2&0&2&6&5&14&12&20&18&6&12\\
$P_{9, 179}$&(6, 3)&0&0&1&2&8&1&15&14&11&27&29&32\\
$P_{9, 180}$&$i_{121}$&1&1&3&3&10&8&14&1&14&3&9&24\\
$P_{9, 181}$&$i_{106}$&1&0&3&3&8&8&13&11&6&11&25&33\\
$P_{9, 182}$&$i_{81}$&0&2&2&2&6&5&14&10&1&27&16&12\\
$P_{9, 183}$&(5, 4)&0&0&1&0&0&9&16&0&0&24&0&16\\
$P_{9, 184}$&$i_{110}$&1&0&4&6&0&0&9&3&6&24&28&4\\
$P_{9, 185}$&(5, 4)&0&0&1&0&0&9&16&0&0&24&0&16\\
$P_{9, 186}$&$i_{87}$&0&2&3&1&1&8&9&15&12&1&18&26\\
$P_{9, 187}$&$i_{110}$&1&0&4&6&0&0&9&3&6&24&28&4\\
$P_{9, 188}$&(7, 3)&0&0&0&1&0&12&0&1&0&0&4&1
\end{tabular}
\vskip5mm

\begin{tabular}{r||r||r|r|r|r|r|r|r|r|r|r|r|r}
graph&$-c_2(p)$&\multicolumn{12}{l}{$p=$}\\
&&\hspace{4pt}2&\hspace{4pt}3&\hspace{4pt}5&\hspace{4pt}7&11&13&17&19&23&29&31&37\\\hline
$P_{9, 189}$&(6, 4)&0&0&4&3&1&11&16&0&13&15&9&35\\
$P_{9, 190}$&$i_{104}$&1&0&2&6&1&9&2&9&8&9&28&18
\end{tabular}
\vskip5mm

\noindent Table 6: List of all $\phi^4$ prime ancestors up to loop order 9 and their $c_2$-invariants for the first 12 primes. The names of the graphs are taken from \cite{CENSUS},
the graphs of loop order 9 are only available in \cite{CENSUS1}. $-c_2$ is either 1, 0 $z_2$, $z_3$, $z_4$, (see Eq.\ (\ref{21})), modular with (weight, level), or an unidentified
sequence from Tab.\ 8.
\vskip5mm

\begin{tabular}{r||r||r|r|r|r|r|r||r|r|r|r|r|r}
no.&$-c_2(p)$&\multicolumn{6}{l||}{$p=$}&\multicolumn{6}{l}{loop order $\ell=$}\\
&&\hspace{4pt}2&\hspace{4pt}3&\hspace{4pt}5&\hspace{4pt}7&11&13&\hspace{4pt}3&\hspace{4pt}6&\hspace{4pt}7&\hspace{4pt}8&\hspace{4pt}9&10\\\hline
$i_{1}$&1&1&1&1&1&1&1&1&0&0&0&0&0\\
$i_{2}$&0&0&0&0&0&0&0&0&1&1&2&0&3\\
$i_{3}$&$z_2$&0&1&1&1&1&1&0&0&2&1&3&2\\
$i_{4}$&$z_3$&1&0&4&1&10&1&0&0&1&2&0&1\\
$i_{5}$&$z_4$&0&2&1&6&10&1&0&0&0&1&0&3\\
$i_{6}$&(3,7)&1&0&0&0&5&0&0&0&0&1&2&0\\
$i_{7}$&(3,8)&0&1&0&0&3&0&0&0&0&1&3&5\\
$i_{8}$&(3,12)&0&0&0&2&0&4&0&0&0&0&1&1\\
$i_{9}$&(4,5)&0&2&0&6&10&1&0&0&0&1&3&4\\
$i_{10}$&(4,6)&0&0&1&5&1&12&0&0&0&0&4&2\\
$i_{11}$&(4,7)&1&1&1&0&3&2&0&0&0&0&0&3\\
$i_{12}$&(4,13)&1&2&3&1&7&0&0&0&0&0&1&1\\
$i_{13}$&(4,17)&1&1&1&0&9&7&0&0&0&0&0&4\\
$i_{14}$&(5,4)&0&0&1&0&0&9&0&0&0&0&4&6\\
$i_{15}$&(6,3)&0&0&1&2&8&1&0&0&0&1&1&6\\
$i_{16}$&(6,4)&0&0&4&3&1&11&0&0&0&0&1&7\\
$i_{17}$&(6,7)&0&1&4&0&1&3&0&0&0&0&2&4\\
$i_{18}$&(6,10)&0&1&0&6&2&6&0&0&0&0&0&1\\
$i_{19}$&(7,3)&0&0&0&1&0&12&0&0&0&0&1&1\\
$i_{20}$&(8,2)&0&0&0&1&3&4&0&0&0&0&0&6\\
$i_{21}$&(8,5)&0&0&0&1&7&1&0&0&0&0&0&1
\end{tabular}
\vskip5mm

\noindent Table 7: List of identified $\phi^4$ $c_2$-invariants up to loop order 10 together with the number of their prime ancestors.
The first five $c_2$-invariants are quasi-constant, $i_{6}$ to $i_{21}$ are modular with (weight, level). There are no prime ancestors with  1, 2, 4, or 5 loops.
\vskip5mm

\footnotesize
\noindent
\begin{tabular}{r||r|r|r|r|r|r||r|rrr||r|r|r|r|r|r||r|rr}
no.&\multicolumn{6}{l||}{$p=$}&\multicolumn{2}{l}{$\ell=$}&&no.&\multicolumn{6}{l||}{$p=$}&\multicolumn{2}{l}{$\ell=$}\\
&\hspace{3pt}2&\hspace{3pt}3&\hspace{3pt}5&\hspace{3pt}7&11&13&\hspace{3pt}9&10&&&\hspace{3pt}2&\hspace{3pt}3&\hspace{3pt}5&\hspace{3pt}7&11&13&\hspace{3pt}9&10\\\hline
$i_{22}$&0&0&0&1&9&2&0&1&&$i_{62}$&0&1&3&1&1&9&0&1\\
$i_{23}$&0&0&0&2&6&7&0&1&&$i_{63}$&0&1&3&2&9&1&0&1\\
$i_{24}$&0&0&0&4&5&3&0&1&&$i_{64}$&0&1&3&3&2&10&0&1\\
$i_{25}$&0&0&1&0&0&0&0&1&&$i_{65}$&0&1&3&4&8&4&0&4\\
$i_{26}$&0&0&1&0&7&7&0&1&&$i_{66}$&0&1&3&5&0&12&0&1\\
$i_{27}$&0&0&1&2&3&2&0&1&&$i_{67}$&0&1&3&5&6&9&0&1\\
$i_{28}$&0&0&1&2&6&5&0&3&&$i_{68}$&0&1&3&5&8&9&0&1\\
$i_{29}$&0&0&1&4&6&10&0&1&&$i_{69}$&0&1&4&1&4&0&0&1\\
$i_{30}$&0&0&1&5&6&5&0&3&&$i_{70}$&0&1&4&2&1&0&0&3\\
$i_{31}$&0&0&2&2&5&12&0&1&&$i_{71}$&0&1&4&2&7&11&0&1\\
$i_{32}$&0&0&2&5&2&9&0&1&&$i_{72}$&0&1&4&3&0&8&0&1\\
$i_{33}$&0&0&2&5&9&11&0&1&&$i_{73}$&0&1&4&5&5&7&0&1\\
$i_{34}$&0&0&2&6&4&0&0&1&&$i_{74}$&0&2&0&0&1&8&0&1\\
$i_{35}$&0&0&2&6&10&0&0&1&&$i_{75}$&0&2&0&0&6&4&0&1\\
$i_{36}$&0&0&3&0&9&3&0&2&&$i_{76}$&0&2&0&0&9&2&0&1\\
$i_{37}$&0&0&3&2&1&2&0&2&&$i_{77}$&0&2&0&5&2&5&0&1\\
$i_{38}$&0&0&3&2&8&11&0&1&&$i_{78}$&0&2&0&6&2&0&0&1\\
$i_{39}$&0&0&3&3&2&10&0&1&&$i_{79}$&0&2&0&6&3&9&0&3\\
$i_{40}$&0&0&4&1&0&9&0&2&&$i_{80}$&0&2&2&2&0&9&0&1\\
$i_{41}$&0&0&4&2&2&8&0&1&&$i_{81}$&0&2&2&2&6&5&1&2\\
$i_{42}$&0&0&4&2&9&11&0&1&&$i_{82}$&0&2&2&3&1&2&0&1\\
$i_{43}$&0&0&4&3&0&11&0&1&&$i_{83}$&0&2&2&3&7&2&0&1\\
$i_{44}$&0&1&0&0&2&4&0&1&&$i_{84}$&0&2&2&5&6&4&0&1\\
$i_{45}$&0&1&0&1&4&5&0&1&&$i_{85}$&0&2&2&6&0&6&0&1\\
$i_{46}$&0&1&0&2&1&9&0&1&&$i_{86}$&0&2&3&0&6&5&0&1\\
$i_{47}$&0&1&0&2&6&9&0&1&&$i_{87}$&0&2&3&1&1&8&1&3\\
$i_{48}$&0&1&0&4&0&0&0&2&&$i_{88}$&0&2&3&3&2&0&0&1\\
$i_{49}$&0&1&0&4&4&0&0&1&&$i_{89}$&0&2&4&0&3&3&0&1\\
$i_{50}$&0&1&1&1&10&0&0&1&&$i_{90}$&0&2&4&3&4&7&0&3\\
$i_{51}$&0&1&1&2&0&9&0&1&&$i_{91}$&0&2&4&4&1&8&0&1\\
$i_{52}$&0&1&1&2&10&9&0&1&&$i_{92}$&0&2&4&4&10&12&0&1\\
$i_{53}$&0&1&1&4&6&6&1&0&&$i_{93}$&0&2&4&5&1&5&0&1\\
$i_{54}$&0&1&1&5&6&3&0&1&&$i_{94}$&0&2&4&6&1&8&0&1\\
$i_{55}$&0&1&1&6&0&5&0&1&&$i_{95}$&1&0&0&1&9&7&0&1\\
$i_{56}$&0&1&1&6&2&0&0&1&&$i_{96}$&1&0&0&3&6&2&0&1\\
$i_{57}$&0&1&2&1&9&2&0&1&&$i_{97}$&1&0&1&2&1&11&0&1\\
$i_{58}$&0&1&2&2&9&12&0&1&&$i_{98}$&1&0&1&2&4&9&0&1\\
$i_{59}$&0&1&2&3&4&10&0&1&&$i_{99}$&1&0&1&2&8&5&0&1\\
$i_{60}$&0&1&2&4&7&1&0&1&&$i_{100}$&1&0&1&2&9&8&0&1\\
$i_{61}$&0&1&3&1&0&11&0&1&&$i_{101}$&1&0&1&4&6&9&0&1
\end{tabular}

\noindent
\begin{tabular}{r||r|r|r|r|r|r||r|rrr||r|r|r|r|r|r||r|rr}
no.&\multicolumn{6}{l||}{$p=$}&\multicolumn{2}{l}{$\ell=$}&&no.&\multicolumn{6}{l||}{$p=$}&\multicolumn{2}{l}{$\ell=$}\\
&\hspace{3pt}2&\hspace{3pt}3&\hspace{3pt}5&\hspace{3pt}7&11&13&\hspace{3pt}9&10&&&\hspace{3pt}2&\hspace{3pt}3&\hspace{3pt}5&\hspace{3pt}7&11&13&\hspace{3pt}9&10\\\hline
$i_{102}$&1&0&2&1&10&5&0&1&&$i_{124}$&1&1&4&6&0&8&0&1\\
$i_{103}$&1&0&2&3&5&2&0&1&&$i_{125}$&1&2&0&2&6&5&1&1\\
$i_{104}$&1&0&2&6&1&9&1&1&&$i_{126}$&1&2&0&5&2&0&0&1\\
$i_{105}$&1&0&3&2&2&12&0&4&&$i_{127}$&1&2&0&5&5&8&0&1\\
$i_{106}$&1&0&3&3&8&8&1&5&&$i_{128}$&1&2&0&5&7&0&0&1\\
$i_{107}$&1&0&4&2&4&12&0&1&&$i_{129}$&1&2&1&0&5&10&0&1\\
$i_{108}$&1&0&4&4&1&11&0&1&&$i_{130}$&1&2&1&0&10&6&0&1\\
$i_{109}$&1&0&4&5&4&8&0&1&&$i_{131}$&1&2&1&2&5&1&0&1\\
$i_{110}$&1&0&4&6&0&0&2&5&&$i_{132}$&1&2&1&4&1&0&1&2\\
$i_{111}$&1&1&0&2&9&11&0&1&&$i_{133}$&1&2&1&4&3&1&0&2\\
$i_{112}$&1&1&0&6&2&9&0&1&&$i_{134}$&1&2&1&6&5&2&0&1\\
$i_{113}$&1&1&0&6&4&1&0&1&&$i_{135}$&1&2&1&6&10&8&0&1\\
$i_{114}$&1&1&1&1&1&8&0&5&&$i_{136}$&1&2&2&0&0&4&0&3\\
$i_{115}$&1&1&1&4&8&6&0&1&&$i_{137}$&1&2&2&3&1&8&0&1\\
$i_{116}$&1&1&1&5&2&8&0&1&&$i_{138}$&1&2&2&5&5&2&0&1\\
$i_{117}$&1&1&2&2&9&7&0&1&&$i_{139}$&1&2&2&5&6&4&0&1\\
$i_{118}$&1&1&2&4&0&12&0&3&&$i_{140}$&1&2&3&3&0&3&0&1\\
$i_{119}$&1&1&3&1&3&6&0&1&&$i_{141}$&1&2&3&3&0&9&1&1\\
$i_{120}$&1&1&3&2&1&12&0&2&&$i_{142}$&1&2&4&1&3&2&0&1\\
$i_{121}$&1&1&3&3&10&8&1&5&&$i_{143}$&1&2&4&5&1&6&0&1\\
$i_{122}$&1&1&3&6&8&6&0&1&&$i_{144}$&1&2&4&6&7&9&0&1\\
$i_{123}$&1&1&4&2&7&10&0&1&&$i_{145}$&1&2&4&6&9&10&0&2
\end{tabular}
\vskip5mm

\normalsize
\noindent Table 8: List of $-c_2$ for all unidentified $c_2$-invariants up to loop order 10 with the number of their prime ancestors.

\bibliographystyle{plain}

\begin{thebibliography}{99}
\bibitem{BB} {\bf P. Belkale, P. Brosnan}, {\it Matroids, motives and a conjecture of Kontsevich}, Duke Math.\ Journal, Vol.\ 116, 147-188 (2003).
\bibitem{BEK} {\bf S. Bloch, H. Esnault, D. Kreimer}, {\it On Motives Associated to Graph Polynomials}, Comm.\ Math.\ Phys.\ 267, 181-225 (2006).
\bibitem{BK} {\bf D. Broadhurst, D.  Kreimer}, {\it  Knots and numbers in $\phi^4$ theory to 7 loops and beyond}, Int.\ J. Mod.\ Phys.\ C6, 519-524 (1995).
\bibitem{B3} {\bf D. Broadhurst}, {\it Multiple zeta values and other periods in quantum field theory}, conference talk, Bristol, 4 May 2011.
\bibitem{BROWN} {\bf F. Brown}, {\it On the periods of some Feynman integrals},  arXiv:0910.0114v2 [math.AG], (2009).
\bibitem{BD} {\bf F. Brown, D. Doryn}, {\it Framings for graph hypersurfaces}, arXiv:1301.3056v1 [math.AG] (2013).
\bibitem{BS} {\bf F. Brown, O. Schnetz}, {\it A K3 in $\phi^4$}, Duke Math.\ Journal, vol.\ 161, no.\ 10 (2012).
\bibitem{BSY} {\bf F. Brown, O. Schnetz, K. Yeats}, {\it Properties of $c_2$ invariants of Feynman graphs}, arXiv:1203.0188v2 [math.AG] (2012).
\bibitem{BY} {\bf F. Brown, K. Yeats}, {\it Spanning forest polynomials and the transcendental weight of Feynman graphs}, Commun.\ Math.\ Phys.\ 301, 357-382 (2011).
\bibitem{CY} {\bf F. Chung, C. Yang}, {\it On polynomials of spanning trees}, Ann.\ Comb.\ 4, 13-25 (2000).
\bibitem{DORY1} {\bf D. Doryn}, {\it On one example and one counterexample in counting rational points on graph hypersurfaces}, arXiv:1006.3533v1 [math-AG] (2010).
\bibitem{IZ} {\bf J. Itzykson, J. Zuber}, {\it Quantum Field Theory}, Mc-Graw-Hill, (1980).
\bibitem{KIR} {\bf G. Kirchhoff}, {\it Ueber die Aufl\"osung der Gleichungen, auf welche man bei der Untersuchung der linearen Vertheilung galvanischer Str\"ome gef\"uhrt wird},
Annalen der Physik und Chemie 72, no.\ 12, 497-508 (1847).
\bibitem{KONT} {\bf M. Kontsevich}, Gelfand Seminar talk, Rutgers University, December 8, 1997.
\bibitem{Meyer} {\bf C. Meyer}, {\tt http://enriques.mathematik.uni-mainz.de/cm/}
\bibitem{NAU} {\bf B. McKay}, nauty, {\tt http://cs.anu.edu.au/$\sim$bdm/nauty} version 2.4$\beta$7 (2007).
\bibitem{Erik} {\bf E. Panzer}, {\it On the analytic computation of massless propagators in dimensional regularization}, arXiv:1305.2161v1 [hep-th] (2013).
\bibitem{Sage} {\tt http://www.sagemath.org/}
\bibitem{CENSUS1} {\bf O. Schnetz}, {\it Quantum periods: A census of $\phi^4$-transcendentals (version 1)}, arXiv: 0801.2856v1 [hep-th] (2008).
\bibitem{CENSUS} {\bf O. Schnetz}, {\it Quantum periods: A census of $\phi^4$ transcendentals}, Jour.\ Numb.\ Theory and Phys.\ 4 no.\ 1, 1-48 (2010).
\bibitem{SchnetzFq} {\bf O. Schnetz}, {\it Quantum field theory over $\FF_q$}, Electron.\ J. Comb.\ 18N1:P102, (2011).
\bibitem{STA} {\bf R. Stanley}, {\it Spanning Trees and a Conjecture of Kontsevich}, Ann.\ Comb.\ 2, 351-363 (1998).
\bibitem{STEM} {\bf J. Stembridge}, {\it Counting Points on Varieties over Finite Fields Related to a Conjecture of Kontsevich}, Ann.\ Comb.\ 2, 365-385 (1998).
\bibitem{etamod} {\tt http://lfunctions.org/degree2/degree2hm/eta2/eta2.html}
\end{thebibliography}
\renewcommand\refname{References}

\end{document}